\documentclass[12pt]{amsart}
\usepackage{amsfonts,amssymb,amscd,amsmath,enumerate,verbatim,color}
\usepackage[latin1]{inputenc}
\usepackage{amscd}
\usepackage{latexsym}

\usepackage{graphicx} 

\usepackage{mathptmx}

\def\NZQ{\mathbb}               

\def\ZZ{{\NZQ Z}}
\def\RR{{\NZQ R}}

\def\frk{\mathfrak}               

\def\Phi{{\frk N}}
\def\ab{{\mathbf a}}

\def\eb{{\mathbf e}}
\def\tb{{\mathbf t}}

\def\wb{{\mathbf w}}
\def\xb{{\mathbf x}}
\def\yb{{\mathbf y}}

\def\opn#1#2{\def#1{\operatorname{#2}}} 
\opn\gr{gr}

\def\Ac{{\mathcal A}}
\def\Bc{{\mathcal B}}

\def\Hc{{\mathcal H}}
\def\Sc{{\mathcal S}}

\def\Gc{{\mathcal G}}
\def\Fc{{\mathcal F}}
\def\Oc{{\mathcal O}}
\def\Pc{{\mathcal P}}
\def\Qc{{\mathcal Q}}

\def\Vol{{\textnormal{Vol}}}

\newtheorem{Theorem}{Theorem}[section]
\newtheorem{Lemma}[Theorem]{Lemma}

\newtheorem{Proposition}[Theorem]{Proposition}

\theoremstyle{definition}
\newtheorem{Remark}[Theorem]{Remark}

\newtheorem{Example}[Theorem]{Example}
\newtheorem{Examples}[Theorem]{Examples}

\let\epsilon\varepsilon
\let\phi=\varphi
\let\kappa=\varkappa
\opn\dis{dis}
\opn\height{height}
\opn\dist{dist}
\def\pnt{{\raise0.5mm\hbox{\large\bf.}}}

\opn\Lex{Lex}
\opn\conv{conv}

\begin{document}

\title{Reflexive polytopes arising from bipartite graphs with $\gamma$-positivity associated to interior polynomials}
\author{Hidefumi Ohsugi and Akiyoshi Tsuchiya}
\address{Hidefumi Ohsugi,
	Department of Mathematical Sciences,
	School of Science and Technology,
	Kwansei Gakuin University,
	Sanda, Hyogo 669-1337, Japan} 
\email{ohsugi@kwansei.ac.jp}

\address{Akiyoshi Tsuchiya,
Graduate school of Mathematical Sciences,
University of Tokyo,
Komaba, Meguro-ku, Tokyo 153-8914, Japan} 
\email{akiyoshi@ms.u-tokyo.ac.jp}


\begin{abstract}
In this paper, we introduce polytopes $\Bc_G$ arising from root systems $B_n$ and
 finite graphs $G$, and study their combinatorial and algebraic properties.
In particular, it is shown that $\Bc_G$ is reflexive if and only if $G$ is bipartite.
Moreover, in the case, $\Bc_G$ has a regular unimodular triangulation.
This implies that the $h^*$-polynomial of $\Bc_G$ is palindromic and unimodal when $G$ is bipartite.
Furthermore, we discuss stronger properties, namely the $\gamma$-positivity and the real-rootedness of the $h^*$-polynomials.
In fact, if $G$ is bipartite, then the 
$h^*$-polynomial of $\Bc_G$ is $\gamma$-positive and its $\gamma$-polynomial is given by an interior polynomial (a version of the Tutte polynomial for a hypergraph). 
The $h^*$-polynomial is real-rooted if and only if the corresponding interior polynomial is real-rooted.
From a counterexample to Neggers--Stanley conjecture, we construct a bipartite graph $G$ whose $h^*$-polynomial is not real-rooted but $\gamma$-positive, and
coincides with the $h$-polynomial of a flag triangulation of a sphere.
\end{abstract}

\maketitle

\section*{Introduction}
A {\em lattice polytope} $\Pc  \subset \RR^d$ is a convex polytope all of whose vertices have integer coordinates.
Ardila et al.~ \cite{Ardila} constructed a unimodular triangulation of the convex hull of the roots of the classical root lattices
of type $A_n$, $B_n$, $C_n$ and $D_n$, and
gave an alternative proof for the known growth series of these root lattices
by using the triangulation.

In \cite{HJMsymmetric}, lattice polytopes arising from the root system of type $A_n$ and finite graphs were introduced.
Let $G$ be a finite simple undirected graph on the vertex set $[d]=\{1, \ldots, d\}$ with the edge set $E(G)$.
We denote $\Ac_G \subset \RR^d$ the convex hull of the set 
$$
A(G)
=\{{\bf 0}\}
\cup
\{\pm ({\bf e}_i - {\bf e}_j)  : \{i,j\} \in E(G)\},
$$
where $\eb_i$ is the $i$-th unit coordinate vector in $\RR^d$ and ${\bf 0}$ is the origin of $\RR^d$.
Then $\Ac_G$ is centrally symmetric, i.e., for any facet $\Fc$ of $\Ac_G$, $-\Fc$ is also a facet of $\Ac_G$, and ${\bf 0}$ is the unique (relative) interior lattice point of $\Ac_G$.
Note that, if $G$ is a complete graph, then $\Ac_G$ coincides with 
the convex hull of the roots of the root lattices
of type $A_n$ studied in \cite{Ardila}.
This polytope $\Ac_G$ is called the {\em symmetric edge polytope} of $G$
and several combinatorial properties of $\Ac_G$ are well-studied (\cite{HKMinterlacing, HJMsymmetric, MHNOH}).

In this paper,  we introduce lattice polytopes arising from the root system of type $B_n$ and finite graphs, and study their algebraic and combinatorial properties.
Let $\Bc_G \subset \RR^d$ denote the convex hull of the set 
$$
B(G)
=\{{\bf 0} , \pm {\bf e}_1, \ldots, \pm {\bf e}_d\}
\cup
\{\pm {\bf e}_i \pm {\bf e}_j  : \{i,j\} \in E(G)\}.
$$
Then $\dim \Bc_G =d$, $\Bc_G$ is centrally symmetric and ${\bf 0}$  is the unique interior lattice point of $\Bc_G$.
Note that, if $G$ is a complete graph, then $\Bc_G$ coincides with 
the convex hull of the roots of the root lattices
of type $B_n$ studied in \cite{Ardila}.
Several classes of lattice polytopes arising from graphs have been studied from viewpoints of combinatorics, graph theory, geometric and commutative algebra.
In particular, {\em edge polytopes} give interesting examples in commutative algebra (\cite{HMT2,OHnormal, OHnoregunitri, Koszulbipartite, SVV}).
Note that edge polytopes of bipartite graphs are called {\em root polytopes} and play important roles in the study of generalized permutohedra (\cite{Postnikov})
and interior polynomials (\cite{KalPos}).

There is a strong relation between $\Bc_G$ and edge polytopes.
In fact, one of the key properties of $\Bc_G$ is that $\Bc_G$ is divided into $2^d$ edge polytopes of certain non-simple graphs $\widetilde{G}$
(Proposition~\ref{keylemma}).
This fact helps us to find and show interesting properties of $\Bc_G$.
In Section~\ref{sec:IDP}, by using this fact, we will classify graphs $G$ such that $\Bc_G$ has a unimodular covering (Theorem \ref{OCCforBG}).
We remark that $\Ac_G$ always has a unimodular covering.

On the other hand, the fact that $\Bc_G$ has a unique interior lattice point ${\bf 0}$ leads us to consider when $\Bc_G$ is reflexive.
A lattice polytope $\Pc \subset \RR^d$ of dimension $d$ is called \textit{reflexive} if the origin of $\RR^d$ is a unique lattice point belonging to the interior of $\Pc$ and its dual polytope
\[\Pc^\vee:=\{\yb \in \RR^d  :  \langle \xb,\yb \rangle \leq 1 \ \text{for all}\  \xb \in \Pc \}\]
is also a lattice polytope, where $\langle \xb,\yb \rangle$ is the usual inner product of $\RR^d$.
It is known that reflexive polytopes correspond to Gorenstein toric Fano varieties, and they are related to
mirror symmetry (see, e.g., \cite{mirror,Cox}).
In each dimension there exist only finitely many reflexive polytopes 
up to unimodular equivalence (\cite{Lag})
and all of them are known up to dimension $4$ (\cite{Kre}).
Here two lattice polytopes $\Pc \subset \RR^d$ and $\Pc' \subset \RR^{d'}$ are said to be \textit{unimodularly equivalent} if there exists an affine map from the affine span ${\rm aff} (\Pc)$ of $\Pc$ to the affine span ${\rm aff} (\Pc')$ of $\Pc'$ that maps $\ZZ^d \cap {\rm aff}(\Pc)$ bijectively onto $\ZZ^{d'} \cap {\rm aff} (\Pc')$ and that maps $\Pc$ to $\Pc'$.
Every lattice polytope is unimodularly equivalent to a full-dimensional one.
In Section \ref{sec:ref}, we will classify graphs $G$ such that 
 $\Bc_G$ is a reflexive polytope.
In fact, we will show the following.
\begin{Theorem}
	\label{reflexiveBG}
	Let $G$ be a finite graph. 
	Then the following conditions are equivalent{\rm :}
	\begin{itemize}
		\item[(i)]
		$\Bc_G$ is reflexive and has a regular unimodular triangulation{\rm ;}
		\item[(ii)]
		$\Bc_G$ is reflexive{\rm ;}
		\item[(iii)]
		$G$ is a bipartite graph.
	\end{itemize}
\end{Theorem}

We remark that not all unimodular equivalence classes of reflexive 
polytopes are represented by a polytope of the form $\Bc_G$. Indeed, 
$\Bc_G$ is always centrally symmetric but not all reflexive polytopes 
are centrally symmetric. On the other hand, $\Ac_G$ is always (unimodularly equivalent to) a centrally symmetric reflexive polytope. However, the class of reflexive polytopes $\Ac_G$ is different from that of $\Bc_G$. For example, if $G$ is a complete bipartite graph, then $\Bc_G$ is not unimodularly equivalent to $\Ac_{G'}$ for any graph $G'$.

Next, by characterizing when the toric ideal of $\Bc_G$ has a Gr\"{o}bner basis consisting of squarefree quadratic binomials for a bipartite graph $G$, we can classify graphs $G$ such that 
$\Bc_G$ is a reflexive polytope with a flag regular unimodular triangulation.
In fact,

\begin{Theorem}
	\label{thm:flag}
	Let $G$ be a bipartite graph.
	Then the following conditions are equivalent{\rm :}
	\begin{itemize}
		\item[(i)]
		The reflexive polytope $\Bc_G$ has a flag regular unimodular triangulation{\rm ;}
		
		\item[(ii)] 		Any cycle of $G$ of length $\ge 6$ has a chord
{\rm (}``chordal bipartite graph''{\rm )}.	
	\end{itemize}
\end{Theorem}

Now, we turn to the discussion of the $h^*$-polynomial $h^*(\Bc_G, x)$ of $\Bc_G$.
Thanks to the key property (Proposition \ref{keylemma}), we can compute the  $h^*$-polynomial of $\Bc_G$ in terms of that of edge polytopes of some graphs.
On the other hand, since it is known that the $h^*$-polynomial of a reflexive polytope with a regular unimodular triangulation is palindromic  and unimodal (\cite{BR}), Theorem \ref{reflexiveBG} implies that the $h^*$-polynomial of $\Bc_G$ is palindromic and  unimodal if $G$ is bipartite.
In Section \ref{sec:gamma}, we will show a stronger result, namely for any bipartite graph $G$, the $h^*$-polynomial $h^*(\Bc_G, x)$ is $\gamma$-positive.
The theory of interior polynomials (a version of the Tutte polynomials for hypergraphs) introduced by K\'{a}lm\'{a}n \cite{interior}
and
the theory of generalized permutohedra \cite{OhTransversal, Postnikov}
play important roles.

\begin{Theorem}
	\label{hpolymain}
	Let $G$ be a bipartite graph on $[d]$.
	Then $h^*$-polynomial of the reflexive polytope $\Bc_G$ is
	$$
	h^*(\Bc_G, x)
	=  (x+1)^d I_{\widehat{G}} \left(  \frac{4x}{(x+1)^2} \right),
	$$
	where 
$\widehat{G}$ is a connected bipartite graph 
defined in $ (\ref{widehat})$ later
and
$I_{\widehat{G}}(x)$ is the interior polynomial of $\widehat{G}$.
	In particular, $h^*(\Bc_G, x)$ is $\gamma$-positive.
	Moreover, $h^*(\Bc_G,x)$ is real-rooted if and only if $I_{\widehat{G}}(x)$ is real-rooted.
\end{Theorem}

In addition, we discuss the following relations between interior polynomials and
other important polynomials in combinatorics:
\begin{itemize}
\item
If $G$ is bipartite, then
the interior polynomial of $\widehat{G}$ is described in terms of $k$-matchings of $G$
 (Proposition \ref{kmatching_interior});

\item
If $G$ is a forest, then
the interior polynomial of $\widehat{G}$ coincides with the matching generating polynomial of $G$ (Proposition \ref{prop:forest});

\item
If $G$ is a bipartite permutation graph associated with a poset $P$,
then the interior polynomial of $\widehat{G}$ coincides with the $P$-Eulerian polynomial of $P$ (Proposition \ref{prop:permutation}).
\end{itemize}
By using these results
and a poset appearing in \cite{Stembridge} as a counterexample to Neggers--Stanley conjecture, we will construct an example of a centrally symmetric reflexive polytope such that the $h^*$-polynomial is $\gamma$-positive and not real-rooted (Example \ref{ex:nonreal}).
This $h^*$-polynomial coincides with the $h$-polynomial of a flag triangulation of a sphere (Proposition~\ref{sphere}).
Hence this example is a counterexample to ``Real Root Conjecture'' that has been already disproved by Gal \cite{Gal}.
Finally, inspired by a simple description for the $h^*$-polynomials of symmetric edge polytopes of complete bipartite graphs \cite{HJMsymmetric}, we will compute the $h^*$-polynomial of $\Bc_G$ when $G$ is a complete bipartite graph (Example \ref{ex:complete}).

\subsection*{Acknowledgment}
The authors are grateful to the anonymous referee for his/her careful reading and helpful comments.
The authors were partially supported by JSPS KAKENHI 18H01134 and 16J01549.

\section{A key property of $\Bc_G$ and unimodular coverings}
\label{sec:IDP}
In this section, we see a relation between $\Bc_G$ and  edge polytopes.
First, we recall what edge polytopes are.
Let $G$ be a graph on $[d]$
(only here we do not assume that $G$ has no loops)
with the edge (including loop) set $E(G)$.
Then the {\em edge polytope} $P_G$ of $G$ is 
the convex hull of $\{{\bf e}_i + {\bf e}_j : \{i,j\} \in E(G)\}$.
Note that $P_G$ is a $(0,1)$-polytope if and only if $G$ has no loops.
Given a graph $G$ on $[d]$, let $\widetilde{G}$ be a graph on $[d+1]$
whose edge set is 
$$E(G) \cup \{ \{1, d+1\}, \{2, d+1\}, \ldots, \{d+1,d+1\} \}.$$
Here, $\{d+1,d+1\}$ is a loop (a cycle of length $1$) at $d+1$.
See Figure~\ref{gtilde}.
\begin{figure}[h]
\begin{center}
\includegraphics[width=7cm]{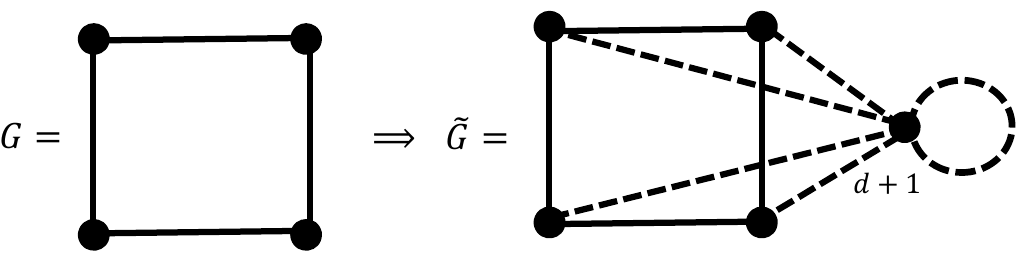} 
\caption{Example of $\widetilde{G}$}
\label{gtilde}
\end{center}
\end{figure}
If $G$ is a bipartite graph with a bipartition
 $V_1 \cup V_2 =[d]$,
 let $\widehat{G}$ be a connected bipartite graph on $[d+2]$
whose edge set is 
\begin{equation}
\label{widehat}
E(\widehat{G}) = E(G) \cup \{ \{i, d+1\}  : i \in V_1\} \cup \{ \{j, d+2\}  : j \in V_2 \cup \{d+1\}\}.
\end{equation}
See Figure~\ref{ghat}.
\begin{figure}[h]
\begin{center}
\includegraphics[width=7cm]{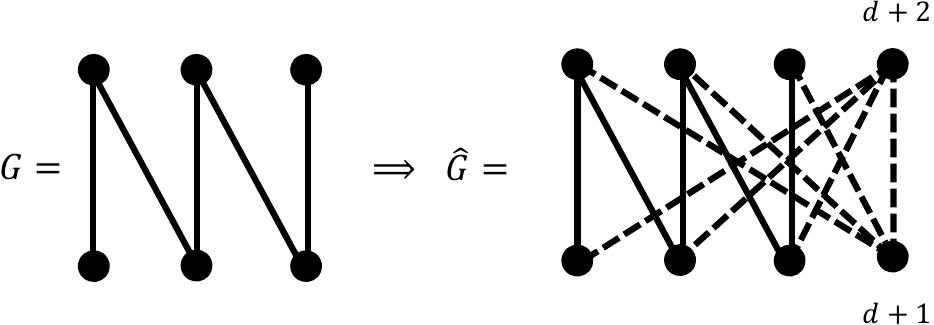} 
\caption{Example of $\widehat{G}$}
\label{ghat}
\end{center}
\end{figure}

Now, we show the key proposition of this paper.
Given $\varepsilon = (\varepsilon_1,\ldots, \varepsilon_d) \in \{-1,1\}^d$,
let ${\mathcal O}_\varepsilon$ denote the closed orthant 
$\{ (x_1,\ldots, x_d) \in \RR^d  : x_i  \varepsilon_i \ge 0 \mbox{ for all } i \in [d]\}$.
For a lattice polytope $\Pc$, let $\Vol(\Pc)$ denote the normalized volume of $\Pc$.
\begin{Proposition}
\label{keylemma}
Work with the same notation as above.
Then we have the following{\rm :}
\begin{itemize}
\item[(a)]
Each $\Bc_G \cap {\mathcal O}_\varepsilon$ is the convex hull of the set
$B(G) \cap {\mathcal O}_\varepsilon$
and unimodularly equivalent to the edge polytope $P_{\widetilde{G}}$ of $\widetilde{G}$.
Moreover, if $G$ is bipartite, then $\Bc_G \cap {\mathcal O}_\varepsilon$ is
unimodularly equivalent to the edge polytope $P_{\widehat{G}}$ of $\widehat{G}$.
In particular, one has $\Vol(\Bc_G)=2^d\Vol(P_{\widetilde{G}})$.
\item[(b)]
The edge polytope of $G$ is a face of $\Bc_G$.
\end{itemize}
\end{Proposition}

\begin{proof}
(a) 
Let $\Pc$ be the convex hull of the set
$B(G) \cap {\mathcal O}_\varepsilon$.
The inclusion $\Bc_G \cap {\mathcal O}_\varepsilon \supset \Pc$ is trivial.
Let $\xb = (x_1,\ldots,x_d) \in \Bc_G \cap {\mathcal O}_\varepsilon$.
Then $\xb = \sum_{i=1}^s \lambda_i \ab_i$, where $\lambda_i >0$,
$\sum_{i=1}^s \lambda_i = 1$, and each $\ab_i$ belongs to $B(G)$.
Suppose that 
the $k$-th component of $\ab_i$ is positive and 
the $k$-th component of $\ab_j$ is negative.
Then $\ab_i$ and $\ab_j$ satisfy 
$$
\ab_i + \ab_j = (\ab_i - \eb_k )+ (\ab_j + \eb_k),
$$
where $\ab_i - \eb_k, \ab_j + \eb_k \in B(G)$.
By using the above relations for $\ab_i + \ab_j$ finitely many times,
we may assume that
the $k$-th component of each vector $\ab_i$ is nonnegative (resp.~nonpositive)
if $x_k \ge 0$ (resp.~$x_k \le 0$).
Then each $\ab_i$ belongs to $B(G) \cap {\mathcal O}_\varepsilon$ and hence
$\xb \in \Pc$.

Next, we show that each $\Bc_G \cap \Oc_{\varepsilon}$ is unimodularly equivalent to the edge polytope $P_{\widetilde{G}}$.
Set $\Qc=\Bc_G \cap \Oc_{(1,\ldots,1)}$. 
It is easy to see that each $\Bc_G \cap \Oc_{\varepsilon}$ is unimodularly equivalent to $\Qc$
for all $\varepsilon$. 
Moreover, one has
$$B(G) \cap  {\mathcal O}_{(1,\ldots,1)}  =  \{{\bf 0}, {\bf e}_1, \ldots,  {\bf e}_d\}
\cup
\{{\bf e}_i + {\bf e}_j   : \{i,j\} \in E(G)\}.$$
Hence $P_{\widetilde{G}}$ 
 is unimodularly equivalent to 
$\Qc \times \{2\}$. In particular, $\Bc_G$ consists of $2^d$ copies of the edge polytope  $P_{\widetilde{G}}$. This implies  $\Vol(\Bc_G)=2^d\Vol(P_{\widetilde{G}})$.
Similarly, if $G$ is bipartite, it follows that 
$P_{\widehat{G}} $ is unimodularly equivalent to $\Qc \times \{1\}^2$.

(b)
The edge polytope $P_G$ of $G$ is the face of $\Bc_G$ with the supporting hyperplane
$
\Hc = 
\{
(x_1,\ldots, x_d) \in \RR^d : x_1 + \cdots + x_d= 2
\}.
$
\end{proof}

Let $\Pc \subset \RR^d$ be a lattice polytope.
A (lattice) \textit{covering} of $\Pc$ is a finite collection $\Delta$ of lattice simplices such that the union of the simplices in $\Delta$ is $\Pc$, i.e., $\Pc=\bigcup_{\delta \in \Delta} \delta$.
A (lattice) \textit{triangulation} of $\Pc$ is a covering $\Delta$ of $\Pc$ such that
\begin{enumerate}
	\item[(i)] every face of a member of $\Delta$ is in $\Delta$, and
	\item[(ii)] any two elements of $\Delta$ intersect in a common (possibly empty) face.
\end{enumerate} 
Note that their interiors may overlap in coverings but not in triangulations.
 A \textit{unimodular simplex} is a lattice simplex whose normalized volume equals $1$.
We now focus on the following properties.
	\begin{itemize}
		\item[(VA)] We say that $\Pc$ is {\em very ample} if for all sufficiently large $k \in \ZZ_{\geq 1}$ and for all $\xb \in k \Pc \cap \ZZ^d$, there exist $\xb_1,\ldots,\xb_k \in \Pc \cap \ZZ^d$ with $\xb=\xb_1+\cdots+\xb_k$.
	\item[(IDP)] We say that $\Pc$ possesses the {\em integer decomposition property} (or is {\em IDP} for short) if for all $k \in \ZZ_{\geq 1}$ and for all $\xb \in k \Pc \cap \ZZ^d$, there exist $\xb_1,\ldots,\xb_k \in \Pc \cap \ZZ^d$ with $\xb=\xb_1+\cdots+\xb_k$.
	\item[(UC)] We say that $\Pc$ has a {\rm unimodular covering} if $\Pc$ admits a lattice covering consisting of unimodular simplices.
	\item[(UT)] We say that $\Pc$ has a {\rm unimodular triangulation} if $\Pc$ admits a lattice triangulation consisting of unimodular simplices.
\end{itemize}
	 These properties satisfy the implications
\[
{\rm(UT)} \Rightarrow {\rm(UC)} \Rightarrow {\rm(IDP)} \Rightarrow {\rm(VA)},
\]
see, e.g., \cite[Section 4.2]{binomialideals}.
On the other hand, it is known that the opposite implications are false. However, for edge polytopes, the first three properties are equivalent (\cite{CHHH, OHnormal, SVV}).
We say that a graph $G$ satisfies the {\em odd cycle condition}
 if, for any two odd cycles $C_1$ and $C_2$ that belong to
the same connected component of $G$
and have no common vertices,
there exists an edge $\{i,j\}$ of $G$ such that $i$ is a vertex of $C_1$
and $j$ is a vertex of $C_2$.

\begin{Proposition}[\cite{CHHH, OHnormal, SVV}]
\label{edge_polytope}
Let $G$ be a finite (not necessarily simple) graph. 
Suppose that there exists an edge $\{i,j\}$ of $G$ 
whenever $G$ has loops at $i$ and $j$ with $i \neq j$.
Then the following conditions are equivalent{\rm :}
\begin{itemize}
\item[(i)]
$P_G$ has a unimodular covering{\rm ;}
\item[(ii)]
$P_G$ is IDP{\rm ;}
\item[(iii)]
$P_G$ is very ample{\rm ;}
\item[(iv)]
$G$ satisfies the odd cycle condition.
\end{itemize}

\end{Proposition}

We now show that the same assertion holds for $\Bc_G$.
Namely, we prove the following.
\begin{Theorem}
	\label{OCCforBG}
	Let $G$ be a finite simple graph. 
	Then the following conditions are equivalent{\rm :}
	\begin{itemize}
		\item[(i)]
		$\Bc_G$ has a unimodular covering{\rm ;}
		\item[(ii)]
		$\Bc_G$ is IDP{\rm ;}
		\item[(iii)]
		$\Bc_G$ is very ample{\rm ;}
		\item[(iv)]
		$G$ satisfies the odd cycle condition.
	\end{itemize}	
\end{Theorem}

Before proving Theorem \ref{OCCforBG}, we recall a well-known result.
\begin{Lemma}\label{face}
Every face of a very ample polytope is very ample.
\end{Lemma}
\begin{proof}
	Let $\Pc \subset \RR^d$ be a lattice polytope of dimension $d$.
Assume that there exists a facet $\Fc$ of $\Pc$ such that $\Fc$ is not very ample.
Then we should show that $\Pc$ is not very ample.
Let $\Hc \subset \RR^d$ be the hyperplane in $\RR^d$ defined as
$\Hc=\{ \xb \in \RR^d \mid \langle \xb, \ab \rangle=b \}$,
with some lattice point $\ab \in \ZZ^d$ and some integer $b$
such that $\Fc=\Hc \cap \Pc$ and $\Pc \subset \Hc^{(+)}$,
where 
$\Hc^{(+)}=\{ \xb \in \RR^d \mid \langle\xb, \ab \rangle \leq b \}$.
Since $\Fc$ is not very ample, there exist a positive integer $k$ and a lattice point $\xb \in k \Fc \cap \ZZ^d$ such that there are no lattice points $\xb_1, \ldots, \xb_k$ in $\Fc$ with $\xb=\xb_1+\cdots+\xb_k$. 
Suppose that there are lattice points $\yb_1, \ldots, \yb_k$ in $\Pc$ with $\xb=\yb_1+\cdots+\yb_k$.  
Then for each $\yb_i$, one has $\langle \yb_i, \ab \rangle \leq b$.
Since $\langle \xb, \ab \rangle = kb$, it follows that
for each $i$, $\langle \yb_i, \ab \rangle=b$, hence, $\yb_i \in \Fc$.
This is a contradiction. 
Therefore, $\Pc$ is not very ample.
\end{proof}
\begin{Remark}
	Similarly, we can prove that every face of an IDP polytope is IDP. Moreover, it is clear that if a lattice polytope has a unimodular triangulation (resp. a unimodular covering), then so does any face.
\end{Remark}
Now, we prove Theorem \ref{OCCforBG}.
\begin{proof}[Proof of Theorem \ref{OCCforBG}]
First, implications (i) $\Rightarrow$ (ii) $\Rightarrow$ (iii) hold in general.

(iii) $\Rightarrow$ (iv):
Suppose that $G$ does not satisfy the odd cycle condition.
By Proposition~\ref{edge_polytope}, the edge polytope $P_G$ of $G$ is not very ample.
Since $P_G$ is a face of $\Bc_G$ by Proposition~\ref{keylemma} (b) and Lemma \ref{face}, $\Bc_G$ is not very ample.

(iv) $\Rightarrow$ (i):
Suppose that $G$ satisfies the odd cycle condition.
Then so does $\widetilde{G}$.
Hence Proposition~\ref{edge_polytope} guarantees that
$P_{\widetilde{G}}$ has a unimodular covering.
By Proposition~\ref{keylemma} (a), 
$\Bc_G$ has a unimodular covering. 
\end{proof}

\begin{Example}
Let $G$ be a graph in Figure \ref{star}.
\begin{figure}[h]
\begin{center}
\includegraphics[width=3cm]{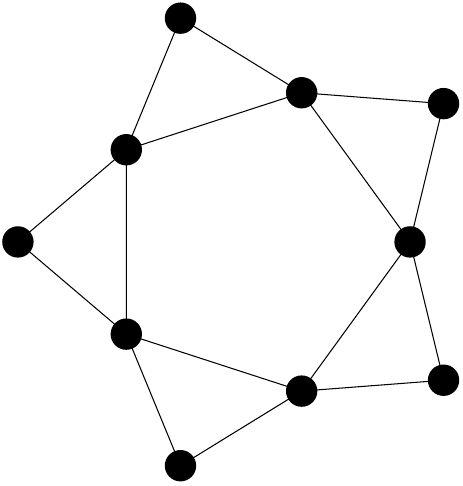} 
\caption{A graph in \cite{OHnoregunitri}}
\label{star}
\end{center}
\end{figure}
Since $G$ satisfies the odd cycle condition, $\Bc_G$ has a unimodular covering.
However, since the edge polytope $P_G$ has no regular unimodular triangulations
(\cite{OHnoregunitri}),
so does $\Bc_G$ by Proposition~\ref{keylemma} (b).
We do not know whether $\Bc_G$ has a (nonregular) unimodular triangulation or not.
\end{Example}


\section{Reflexive polytopes and flag triangulations of $\Bc_G$}
\label{sec:ref}
In the present section, we classify graphs $G$ such that
\begin{itemize}
	\item
	$\Bc_G$ is a reflexive polytope.
	\item
	$\Bc_G$ is a reflexive polytope with a flag regular unimodular triangulation.
\end{itemize}
In other words, we prove Theorems \ref{reflexiveBG} and \ref{thm:flag}.
First, we see some examples that $\Bc_G$ is reflexive.
\begin{Examples}
\label{ex}
	(a)
	If $G$ is an empty graph, then  $\Bc_G$ is a {\em cross polytope}.
	
	(b) 
	Let $G$ be a complete graph with $2$ vertices.
	Then $\Bc_G \cap \ZZ^2$ is the column vectors of the matrix
	$$
	\begin{pmatrix}
	0 & 1 & 0 & 1 & 1 &-1 & 0 & -1 & -1\\
	0 & 0 & 1 & 1 &-1 & 0 &-1 & -1 & 1
	\end{pmatrix},
	$$
	and $\Bc_G$ is a reflexive polytope having a regular unimodular triangulation.
See Figure~\ref{ex21}.
\begin{figure}
\begin{center}
\includegraphics[width=3cm]{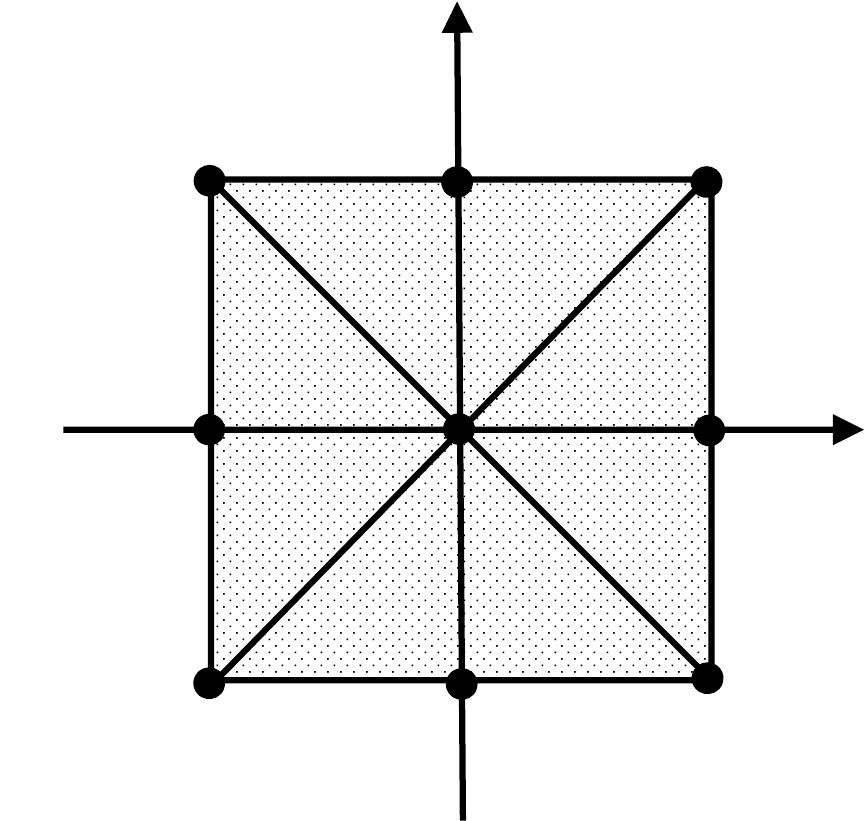} 
\caption{Regular unimodular triangulation of $\Bc_G$ in Examples \ref{ex} (b)}
\label{ex21}
\end{center}
\end{figure}	Since the matrix
	$$
	\begin{pmatrix}
	1 & 0 & 1 & 1 \\
	0 & 1 & 1 &-1 
	\end{pmatrix}
	$$
	is not unimodular, we cannot apply \cite[Lemma 2.11]{CSC} to show this fact.
\end{Examples}
In order to show that a lattice polytope is reflexive, we can use an algebraic technique on Gr\"obner bases.
We recall basic materials and notation on toric ideals.
Let $K[{\bf t}^{\pm1}, s] 
= K[t_{1}^{\pm1}, \ldots, t_{d}^{\pm1}, s]$
be the Laurent polynomial ring in $d+1$ variables over a field $K$. 
If $\ab = (a_{1}, \ldots, a_{d}) \in \ZZ^{d}$, then
${\bf t}^{\ab}s$ is the Laurent monomial
$t_{1}^{a_{1}} \cdots t_{d}^{a_{d}}s \in K[{\bf t}^{\pm1}, s]$. 
Let  $\Pc \subset \RR^{d}$ be a lattice polytope and $\Pc \cap \ZZ^d=\{\ab_1,\ldots,\ab_n\}$.
Then, the \textit{toric ring}
of $\Pc$ is the subring $K[\Pc]$ of $K[{\bf t}^{\pm1}, s] $
generated by
$\{\tb^{\ab_1}s ,\ldots,\tb^{\ab_n}s \}$ over $K$.
We regard $K[\Pc]$ as a graded ring by setting each $\text{deg } t_i =0$ and $\text{deg } s=1$.
Note that, if two lattice polytopes are unimodularly equivalent, then their toric rings are isomorphic as graded rings.
Let $K[\xb]=K[x_1,\ldots, x_n]$ denote the polynomial ring in $n$ variables over $K$.
The \textit{toric ideal} $I_{\Pc}$ of $\Pc$ is the kernel of the surjective homomorphism $\pi : K[\xb] \rightarrow K[\Pc]$
defined by $\pi(x_i)=\tb^{\ab_i}s$ for $1 \leq i \leq n$.
It is known that $I_{\Pc}$ is generated by homogeneous binomials.
See, e.g., \cite[Chapter~3]{binomialideals} and \cite[Chapter~4]{sturmfels1996}.
Let $<$ be a monomial order on $K[{\bf x}]$ and 
${\rm in}_{<}(I_{\Pc})$ the initial ideal of $I_{\Pc}$ with respect to $<$.
The initial ideal ${\rm in}_{<}(I_{\Pc})$ is called {\em squarefree} 
(resp.~{\em quadratic}) if 
${\rm in}_{<}(I_{\Pc})$ is generated by squarefree (resp.~{quadratic}) monomials.
We recall a relation between initial ideal ${\rm in}_{<}(I_{\Pc})$ and a triangulation of $\Pc$.
Set
\[
\Delta(\Pc, <)= \left\{ \conv (S) : S \subset \Pc \cap \ZZ^d, \prod_{\ab_i \in S} x_i \notin \sqrt{ {\rm in}_{<}(I_{\Pc})} \right\}.
\]
Then we have 
\begin{Proposition}[{\cite[Theorem 8.3]{sturmfels1996}}]
Let $\Pc \subset \RR^d$ be a lattice polytope and $<$ a monomial order on $K[{\bf x}]$.
Then $\Delta(\Pc, <)$ is a regular triangulation of $\Pc$.
\end{Proposition}
Here we say that a lattice triangulation $\Delta$ of a lattice polytope $\Pc \subset \RR^d$ is \textit{regular} if it arises as the projection of the lower hull of a lifting of the lattice points of $\Pc$ into $\RR^{d+1}$.
A simplicial complex is called \textit{flag}
 if all minimal non-faces contain only two elements. We say that a lattice triangulation is \textit{flag} if it is flag as a simplicial complex.
 Then we have
 \begin{Proposition}[{\cite[Corollary 8.9]{sturmfels1996}}]
 	Let $\Pc \subset \RR^d$ be a lattice polytope and $<$ a monomial order on $K[{\bf x}]$. Suppose that any lattice point in $\ZZ^{d+1}$ is a linear integer combination of the lattice points in $\Pc \times \{1\}$.
 	Then  $\Delta(\Pc, <)$ is unimodular (resp. flag unimodular) if and only if ${\rm in}_{<}(I_{\Pc})$ is squarefree (resp. squarefree quadratic).
 \end{Proposition}

By definition, a lattice polytope $\Pc \subset \RR^d$ is reflexive if and only if the lattice distance between the origin ${\bf 0}$ and all affine hyperplanes generated by facets of $\Pc$ equals $1$.
Hence if  ${\bf 0}$ is contained in the interior of $\Pc$, and  $\Pc$ has a unimodular triangulation such that ${\bf 0}$ is a vertex of any maximal simplex of the triangulation, then $\Pc$ is reflexive.
Now, we introduce an algebraic technique to show that a lattice polytope is reflexive.

\begin{Lemma}[{\cite[Lemma 1.1]{HMOS}}]
	\label{HMOS}
	Let $\Pc \subset \RR^d$ be a lattice polytope of dimension $d$ such that the origin of $\RR^d$ is contained 
	in its interior and $\Pc \cap \ZZ^d=\{\ab_1,\ldots,\ab_n \}$.
	Suppose that 
$\ZZ^d = \{\sum_{i=1}^n z_i \ab_i : z_i \in \ZZ \}$ 
and there exists an ordering of the variables $x_{i_1} < \cdots < x_{i_n}$ for which $\ab_{i_1}= \mathbf{0}$ such that the initial ideal $\textnormal{in}_{<}(I_{\Pc})$ of $I_{\Pc}$ with respect to the reverse lexicographic order $<$ on $K[\xb]$
	induced by the ordering is squarefree.
	Then $\Pc$ is reflexive and has a regular unimodular triangulation.
\end{Lemma}
\begin{proof}
	The assertion follows since ${\bf 0}$ is contained in the interior of $\Pc$ and  since the triangulation $\Delta(\Pc, <)$ is unimodular and ${\bf 0}$ is a vertex of any maximal simplex of the triangulation.
	\end{proof}

By using this technique, several families of reflexive polytopes with regular unimodular triangulations are constructed in \cite{twin,HMOS,HMT1,HTomega,HTperfect,HTos,harmony}.
In order to apply Lemma \ref{HMOS} to show Theorem \ref{reflexiveBG}, we see a relation between the toric ideal of $\Bc_{G}$ and that of $P_{\widetilde{G}}$. 
Let $G$ be a simple graph on $[d]$ with edge set $E(G)$ and let $R_G$ denote the polynomial ring
in $2d+1+4 |E(G)|$ variables 
$$
z, \ \ x_{i+}, x_{i-} \ \  (1 \leq i  \leq d), \ \ 
y_{ij++}, y_{ij--}, y_{ij+-}, y_{ij-+} \ \  (\{i,j\} \in E(G))
$$
over a field $K$.
Then the toric ideal $I_{\Bc_G}$ of $\Bc_G$ is the kernel of a ring homomorphism 
 $\pi: R_G \rightarrow K[t_1^\pm,\ldots,t_d^\pm, s]$
defined by
$\pi(z) = s$, 
$\pi(x_{i+}) = t_i s$, 
$\pi(x_{i-}) = t_i^{-1} s$, 
$\pi(y_{ij++}) = t_i t_j s$, 
$\pi(y_{ij--}) = t_i^{-1} t_j^{-1} s$, 
$\pi(y_{ij+-}) = t_i t_j^{-1} s$, 
and
$\pi(y_{ij-+}) = t_i^{-1} t_j s$.
Let $S_G$ denote the subring of $R_G$ generated by
the $d+1+ |E(G)|$ variables 
$$
z, \ \  x_{i+} \ \ (1 \leq i  \leq d),
\ \  y_{ij++} \ \ (\{i,j\} \in E(G)).
$$
Then the toric ideal $I_{P_{\widetilde{G}}}$ of $P_{\widetilde{G}}$
is the kernel of $\pi|_{S_G}$.
For each $\varepsilon = (\varepsilon_1,\ldots, \varepsilon_d)  \in \{-1,1\}^d$, 
we define a ring homomorphism $\varphi_\varepsilon: S_G \rightarrow R_G$
by $\varphi_\varepsilon(x_{i+}) =x_{i \alpha}$ and  $\varphi_\varepsilon(y_{ij++}) =y_{ij \alpha \beta}$ where 
$\alpha$ is the sign of $\varepsilon_i$ and $\beta$ is the sign of $\varepsilon_j$.
In particular, $\varphi_{(1,\ldots,1)}: S_G \rightarrow R_G$ is an inclusion map.

\begin{Lemma}
\label{GBlemma}
Let $\Gc$ be a Gr\"obner basis of $I_{P_{\widetilde{G}}}$ with respect to a reverse lexicographic order
$<_S$ on $S_G$ such that $z < \{x_{i+}\} < \{y_{ij++}\}$.
Let $<_R$ be a reverse lexicographic order such that
 $z < \{x_{i+}, x_{i-}\} < \{y_{ij++}, y_{ij--}, y_{ij+-}, y_{ij-+}\}$
and that {\rm (i)}
$
\varphi_\varepsilon(x_{i+}) <_R \varphi_\varepsilon(x_{j+})  
\mbox{ if }
x_{i+} <_S x_{j+}
$
and {\rm (ii)}
$
\varphi_\varepsilon(y_{ij++}) <_R \varphi_\varepsilon(y_{k\ell++})  
\mbox{ if }
y_{ij++} <_S y_{k\ell++}
$
for all $\varepsilon   \in \{-1,1\}^d$.
Then
\begin{eqnarray*}
\Gc' &=& 
\left(\bigcup_{\varepsilon   \in \{-1,1\}^d} \varphi_\varepsilon(\Gc)  \right)
\cup
\{
\underline{ x_{i \alpha} y_{i j \beta \gamma} } - x_{j \gamma} z  : \{i,j\} \in E(G),
\alpha \neq \beta
\}\\
&\cup&
\{
\underline{ y_{i j \alpha \gamma} y_{i k \beta \delta} } - x_{j \gamma} x_{k \delta}  : \{i,j\} ,\{i,k\}\in E(G),
\alpha \neq \beta
\}
\cup
\{
\underline{ x_{i+} x_{i-}} - z^2  : 1 \le i \le d
\}
\end{eqnarray*}
is a Gr\"obner basis of $I_{\Bc_G}$ with respect to $<_R$,
where the underlined monomial is the initial monomial of each binomial.
{\rm (}Here we identify $y_{ij \alpha \beta}$ with $y_{ji  \beta \alpha}$.{\rm )}
In particular, if ${\rm in}_{<_S}( I_{P_{\widetilde{G}}})$ is squarefree {\rm (}resp. quadratic{\rm )},
then so is ${\rm in}_{<_R} (I_{\Bc_G})$.
\end{Lemma}

\begin{proof}
It is easy to see that $\Gc'$ is a subset of $I_{\Bc_G}$.
Assume that $\Gc'$ is not a Gr\"obner basis of $I_{\Bc_G}$
 with respect to $<_R$.
Let ${\rm in}(\Gc') = \left< {\rm in}_{<_R} (g)  : g \in \Gc' \right>$.
By \cite[Theorem~3.11]{binomialideals}, there exists a non-zero irreducible homogeneous binomial $f = u-v \in I_{\Bc_G}$
such that neither $u$ nor $v$ belongs to ${\rm in}(\Gc')$.
Since both $u$ and $v$ are divided by none of $x_{i+} x_{i-},$
$x_{i \alpha} y_{i j \beta \gamma}$, $y_{i j \alpha \gamma} y_{i k \beta \delta} $ ($\alpha \neq \beta$), 
they are of the form
$$
u= \prod_{i \in I} ({x_{i \alpha_i}})^{u_i} \prod_{\{j,k\} \in E_1} (y_{j k \alpha_j \alpha_k})^{u_{jk}},
\ \ \ 
v= \prod_{i \in I'} (x_{i \alpha_i'})^{v_i} \prod_{\{j,k\} \in E_2} (y_{j k \alpha_j' \alpha_k'})^{v_{jk}},
$$
where $I, I' \subset [d]$, $E_1, E_2 \subset E(G)$ and $0 < u_i, u_{jk}, v_i, v_{jk} \in \ZZ$.
Since $f$ belongs to $I_{\Bc_G}$, the exponent of $t_\ell$ in $\pi(u)$ and $\pi(v)$ are the same.
Hence, one of $x_{\ell \alpha_\ell}$ and $y_{\ell m \alpha_\ell \alpha_m}$ appears in $u$ if and only if 
one of $x_{\ell \alpha_\ell'}$ and $y_{\ell n \alpha_\ell' \alpha_n'}$ appears in $v$
with $\alpha_\ell = \alpha_\ell'$.
Let $\varepsilon$ be a vector in $\{-1,1\}^d$  such that 
the sign of the $i$-th component of $\varepsilon$ is $\alpha_i$
if one of $x_{i \alpha_i}$ and $y_{i j \alpha_i \alpha_j}$ appears in $u$.
Then $f$ belongs to the ideal $\varphi_\varepsilon(I_{P_{\widetilde{G}}})$.
Let $f' \in I_{P_{\widetilde{G}}}$ be a binomial such that $\varphi_\varepsilon  (f') = f$.
Since $\Gc$ is a Gr\"obner basis of $I_{P_{\widetilde{G}}}$, 
there exists a binomial $g \in \Gc$ whose initial monomial ${\rm in}_{<_R} (g)$ divides 
the one of the monomials in $f'$. 
By the definition of $<_R$, 
we have ${\rm in}_{<_R} (\varphi_\varepsilon(g) ) =\varphi_\varepsilon ({\rm in}_{<_R} (g) )$.
Hence ${\rm in}_{<_R} (\varphi_\varepsilon(g) )$ divides one of the monomials in $f = \varphi_\varepsilon  (f') $.
This is a contradiction.
\end{proof}

Using this Gr\"obner basis with respect to a reverse lexicographic order,
we verify which $\Bc_G$ is a reflexive polytope.
Namely, we prove Theorem \ref{reflexiveBG}.

\begin{proof}[Proof of Theorem \ref{reflexiveBG}]
The implication (i) $\Rightarrow$ (ii) is trivial.

(ii) $\Rightarrow$ (iii):
Suppose that $G$ is not bipartite.
Let $G_1, \ldots, G_s$ be the connected components of $G$
and let $d_i$ be the number of vertices of $G_i$.
In particular, we have $d = \sum_{i=1}^s d_i$.
Since $G$ is not bipartite, we may assume that $G_1$ is not bipartite.
Let $\wb = \sum_{i=1}^s \wb_i$, where 
$\wb_i = \sum_{k=1}^\ell \eb_{p_k} \in \RR^d$
if $G_i$ is a non-bipartite graph on the vertex set $\{p_1, \ldots, p_\ell\}$,
and $\wb_i = \sum_{k=1}^\ell 2 \eb_{p_k} \in \RR^d$
if $G_i$ is a bipartite graph whose vertices are divided into 
two independent sets $\{p_1, \ldots, p_\ell\}$ and $\{q_1, \ldots, q_m\}$.
It then follows that 
$
\Hc = 
\{
\xb \in \RR^d : \wb \cdot \xb = 2
\}
$
is a supporting hyperplane of $\Bc_G$
and the corresponding face $\Fc_\wb = \Bc_G \cap \Hc$  is the convex hull of
$H=\bigcup_{i=1}^s H_i$, where 
$H_i=
\{{\bf e}_u + {\bf e}_v : \{u,v\} \in E(G_i)\}
$
if $G_i$ is not bipartite, and
$H_i=
\{{\bf e}_u + {\bf e}_v : \{u,v\} \in E(G_i)\} \cup \{\eb_{p_1}, \ldots, \eb_{p_\ell}\}
\cup \{{\bf e}_{p_k} - {\bf e}_v : \{p_k,v\} \in E(G_i)\}
$
if $G_i$ is a bipartite graph with $\wb_i = \sum_{k=1}^\ell 2 \eb_{p_k} \in \RR^d$.
We will show that $\Fc_\wb$ is a facet of $\Bc_G$.
The convex hull  of $\{{\bf e}_u + {\bf e}_v : \{u,v\} \in E(G_i)\}$
 is the edge polytope $P_{G_i}$ of $G_i$ and 
it is known \cite[Proposition~1.3]{OHnormal} that
$$
\dim P_{G_i}
=
\left\{
\begin{array}{cc}
d_i-1 & \mbox{if } G_i \mbox{ is not bipartite,}\\
d_i-2  & \mbox{otherwise.}
\end{array}
\right.
$$
If $G_i$ is not bipartite, then the dimension of $\conv(H_i) = P_{G_i}$ is $d_i-1$.
If $G_i$ is bipartite, then 
$d_i -2  = \dim P_{G_i} < \dim \conv(H_i)$ since
$P_{G_i} \subset \conv(H_i)$
and a hyperplane 
$
\Hc = 
\{
(x_1,\ldots, x_d) \in \RR^d : x_1 + \cdots + x_d= 2
\} 
$ satisfies $\eb_{p_1} \notin \Hc \supset P_{G_i}$.
Hence the dimension of the face $\Fc_\wb$ is at least $s -1 + \sum_{i=1}^s (d_i-1) = d-1$,
i.e., $\Fc_\wb$ is a facet of $\Bc_G$.
Since $G_1$ is not bipartite, we have $\frac{1}{2} \cdot \wb \notin \ZZ^d$.
Thus $\Bc_G$ is not reflexive.

(iii) $\Rightarrow$ (i):
Suppose that $G$ is bipartite.
Let $<_S$ and $<_R$ be any reverse lexicographic orders satisfying the condition in 
 Lemma~\ref{GBlemma}.
It is known \cite[Theorem~5.24]{binomialideals} that any triangulation of the edge polytope of a bipartite graph is unimodular.
By \cite[Corollary~8.9]{sturmfels1996}, the initial ideal of the toric ideal of $P_{\widehat{G}}$ with respect to $<_S$ is squarefree.
Thanks to Lemmas~\ref{HMOS} and \ref{GBlemma}, we have the desired conclusion.
\end{proof}

We now give a theorem on quadratic Gr\"obner bases of $I_{\Bc_G}$ when $G$ is bipartite.
This theorem implies that Theorem \ref{thm:flag}.
The same result is known for edge polytopes (\cite{Koszulbipartite}).

\begin{Theorem}
\label{kb}
Let $G$ be a bipartite graph.
Then the following conditions are equivalent{\rm :}
\begin{itemize}
\item[(i)]
The toric ideal $I_{\Bc_G}$ of $\Bc_G$ has a squarefree quadratic 
initial ideal\\
{\rm (}i.e., $\Bc_G$ has a flag regular unimodular triangulation{\rm );}

\item[(ii)]
The toric ring $K[\Bc_G]$ of $\Bc_G$ is a Koszul algebra{\rm ;}

\item[(iii)]
The toric ideal $I_{\Bc_G}$ of $\Bc_G$ is generated by quadratic binomials{\rm ;}

\item[(iv)]
Any cycle of $G$ of length $\ge 6$ has a chord
{\rm (}``chordal bipartite graph''{\rm )}.
\end{itemize}

\end{Theorem}

\begin{proof}
Implications (i) $\Rightarrow$ (ii) $\Rightarrow$ (iii) hold in general.

 (iii) $\Rightarrow$ (iv):
Suppose that $G$ has a cycle of length $\ge 6$ without chords.
By the theorem in \cite{Koszulbipartite}, the toric ideal  of $P_G$ is not generated by 
quadratic binomials.
Since the edge polytope $P_G$ is a face of $\Bc_G$,
the toric ring $K[P_G]$ is 
a combinatorial pure subring \cite{cpure} of $K[\Bc_G]$.
Hence $I_{\Bc_G}$ is not generated by quadratic binomials.

(iv) $\Rightarrow$ (i):
Suppose that any cycle of $G$ of length $\ge 6$ has a chord.
By Lemma~\ref{GBlemma}, it is enough to show that the initial ideal of 
$I_{P_{\widehat{G}}}$ is squarefree and quadratic with respect to a reverse lexicographic order $<_S$
such that  $z < \{x_{i+}\} < \{y_{k \ell ++}\}$.
Let $A = (a_{ij})$ be the incidence matrix of $G$ whose rows are indexed by $V_1$
and whose columns are indexed by $V_2$.
Then the incidence matrix of $\widehat{G}$ is 
$$A' = 
\left(
\begin{array}{cccc}
 & & & 1\\
 & A & & \vdots\\
 & & & 1\\
1 & \cdots & 1 & 1
\end{array}
\right)
$$
By the same argument as in the proof of the theorem in \cite{Koszulbipartite}, we may assume that 
$A'$ contains no submatrices 
$\left(
\begin{array}{cc}  
1 & 1\\
1 & 0
\end{array}
\right)
$ if we permute the rows and columns of $A$ in $A'$. 
Each quadratic binomial in $I_{P_{\widehat{G}}}$ corresponds to a submatrix 
$\left(
\begin{array}{cc}  
1 & 1\\
1 & 1
\end{array}
\right)
$ of $A'$.
The proof of the theorem in \cite{Koszulbipartite} guarantees that 
the initial ideal is squarefree and quadratic if the initial monomial of each quadratic binomial corresponds 
to $\left(
\begin{array}{cc}  
  & 1\\
1 &  
\end{array}
\right)$.
It is easy to see that there exists a such reverse lexicographic order
which satisfies $z < \{x_{i+}\} < \{y_{k \ell ++}\}$.
\end{proof}


\section{$\gamma$-positivity and real-rootedness of 
the $h^*$-polynomial of ${\mathcal B}_{G}$}
\label{sec:gamma}
In this section, we study the $h^*$-polynomial of $\Bc_G$ for a graph $G$.
First, we recall what $h^*$-polynomials are.
Let $\Pc \subset \RR^d$ be a lattice polytope of dimension $d$.
Given a positive integer $n$, we define
$$L_{\Pc}(n)=|n \Pc \cap \ZZ^d|.$$
The study on $L_{\Pc}(n)$ originated in Ehrhart \cite{Ehrhart} who proved that $L_{\Pc}(n)$ is a polynomial in $n$ of degree $d$ with the constant term $1$.
We say that $L_{\Pc}(n)$ is the \textit{Ehrhart polynomial} of $\Pc$.
The generating function of the lattice point enumerator, i.e., the formal power series
$$\text{Ehr}_\Pc(x)=1+\sum\limits_{k=1}^{\infty}L_{\Pc}(k)x^k$$
is called the \textit{Ehrhart series} of $\Pc$.
It is well known that it can be expressed as a rational function of the form
$$\text{Ehr}_\Pc(x)=\frac{h^*(\Pc,x)}{(1-x)^{d+1}}.$$
 The polynomial $h^*(\Pc,x)$ is a polynomial in $x$ of degree at most $d$ with nonnegative integer coefficients (\cite{Stanleynonnegative}) and it
is called
the \textit{$h^*$-polynomial} (or the \textit{$\delta$-polynomial}) of $\Pc$. 
Moreover, one has $\Vol(\Pc)=h^*(\Pc,1)$. 
Refer the reader to \cite{BeckRobins} for the detailed information about Ehrhart polynomials and $h^*$-polynomials.

Thanks to Proposition \ref{keylemma} (a), we give a formula for $h^*$-polynomial of $\Bc_G$ in terms of that of edge polytopes of some graphs.
By the following formula, we can calculate the $h^*$-polynomial of $\Bc_G$
if we can calculate each $h^*(P_{\widetilde{H}}, x)$.

\begin{Proposition}
\label{hpolyformula}
	Let $G$ be a graph on $[d]$.
	Then the $h^*$-polynomial of $\Bc_G$ satisfies
	\begin{eqnarray}
	\label{fundamentalF}
	h^*(\Bc_G, x)
	&=&
	\sum_{j=0}^d \ \ 
	2^j (x-1)^{d-j}  
	\sum_{H \in S_j(G)} h^*(P_{\widetilde{H}}, x),
	\end{eqnarray}
	where $S_j(G)$ denote the set of all induced subgraph of $G$
	with $j$ vertices.
\end{Proposition}

\begin{proof}
	By Proposition \ref{keylemma} (a), $\Bc_G$ is divided into $2^d$ lattice polytopes of the form
	$\Bc_G \cap {\mathcal O}_\varepsilon$.
	Each $\Bc_G \cap {\mathcal O}_\varepsilon$ is 
	unimodularly equivalent to $P_{\widetilde{G}}$.
	In addition, the intersection of 
	$\Bc_G \cap {\mathcal O}_\varepsilon$ and $\Bc_G \cap {\mathcal O}_{\varepsilon'}$
	is of dimension $d-1$ if and only if $\varepsilon - \varepsilon' \in \{\pm 2 \eb_1, \ldots, \pm 2 \eb_d\}$.
	If $\varepsilon - \varepsilon' = 2 \eb_k$, then 
	$
	(\Bc_G \cap {\mathcal O}_\varepsilon) \cap (\Bc_G \cap {\mathcal O}_{\varepsilon'})
	= 
	\Bc_G \cap {\mathcal O}_\varepsilon \cap {\mathcal O}_{\varepsilon'}
	\simeq
	\Bc_{G'} \cap {\mathcal O}_{\varepsilon''},
	$
	where $G'$ is the induced subgraph of $G$ obtained by deleting the vertex $k$,
	and $\varepsilon''$ is obtained by deleting the $k$-th component of $\varepsilon$.
	Hence the Ehrhart polynomial $L_{\Bc_G}(n)$ satisfies the following:
	$$
	L_{\Bc_G}(n) = \sum_{j=0}^d \ \ 
	2^j (-1)^{d-j}  
	\sum_{H \in S_j(G)} L_{P_{\widetilde{H}}}(n).
	$$
	Thus the Ehrhart series satisfies
	\begin{eqnarray*}
		\frac{h^*(\Bc_G, x)}{(1-x)^{d+1}}
		&=&
		\sum_{j=0}^d \ \ 
		2^j (-1)^{d-j}  
		\sum_{H \in S_j(G)} 
		\frac{h^*(P_{\widetilde{H}}, x)}{(1-x)^{j+1}},
	\end{eqnarray*}
	as desired.
\end{proof}

Let  $f= \sum_{i=0}^{d}a_i x^i$ be a polynomial with real coefficients and $a_d \neq 0$.
We now focus on the following properties.
\begin{itemize}
	\item[(RR)] We say that $f$ is {\em real-rooted} if all its roots are real.  
	\item[(LC)] We say that $f$ is {\em log-concave} if $a_i^2 \geq a_{i-1}a_{i+1}$ for all $i$.
	\item[(UN)] We say that $f$ is {\em unimodal} if $a_0 \leq a_1 \leq \cdots \leq a_k \geq \cdots \geq a_d$ for some $k$.
\end{itemize}  
If all its coefficients are nonnegative, then these properties satisfy the implications
\[
{\rm(RR)} \Rightarrow {\rm(LC)} \Rightarrow {\rm(UN)}.
\]
On the other hand, the polynomial $f$ is said to be {\em palindromic} if $f(x)=x^df(x^{-1})$.
It is {\em $\gamma$-positive} if there are $\gamma_0,\gamma_1,\ldots,\gamma_{\lfloor d/2\rfloor} \geq 0$ such that $f(x)=\sum_{i \geq 0}
\gamma_i \  x^i (1+x)^{d-2i}$.
The polynomial $\sum_{i \geq 0}\gamma_i \ x^i$ is called {\em $\gamma$-polynomial of $f$}. We can see that a $\gamma$-positive polynomial is real-rooted if and only if its $\gamma$-polynomial had only real roots (\cite[Observation 4.2]{EulerianNumbers}).

By the following proposition, we are interested in connected bipartite graphs.

\begin{Proposition}
 Let $G$ be a bipartite graph and $G_1,\ldots, G_s$ the connected components of $G$.
Then the $h^*$-polynomial of ${\mathcal B}_{G}$ is palindromic, unimodal
and 
	\[
	h^*(\Bc_G,x)=h^*(\Bc_{G_1},x) \cdots h^*(\Bc_{G_s},x).
	\]
\end{Proposition}

\begin{proof}
It is known \cite{hibi} that the $h^*$-polynomial of a lattice polytope $P$ with the interior lattice point ${\bf 0}$  is palindromic 
if and only if $P$ is reflexive.
Moreover, if a reflexive polytope $P$ has a unimodular triangulation, then 
the $h^*$-polynomial of $P$ is unimodal (see \cite{BR}).
It is easy to see that, 
$\Bc_G$ is the free sum of $\Bc_{G_1}, \ldots, \Bc_{G_s}$.
Thus we have a desired conclusion by Theorem~\ref{reflexiveBG}
and \cite[Theorem 1]{Braun}.
\end{proof}

In the rest of the present paper, we discuss the $\gamma$-positivity and the real-rootedness of the $h^*$-polynomial of $\Bc_G$
when $G$ is a bipartite graph.
The edge polytope $P_G$ of a bipartite graph $G$ is called a {\em root polytope}
of $G$, and it is shown  \cite{KalPos} that the $h^*$-polynomial of $P_G$ coincides with the interior polynomial $I_G(x)$ of a hypergraph induced by $G$.
First, we discuss interior polynomials introduced by K\'{a}lm\'{a}n \cite{interior}
and developed in many papers.

A {\em hypergraph} is a pair $\Hc = (V, E)$, where $E=\{e_1,\ldots,e_n\}$ is a finite multiset of non-empty subsets of $V=\{v_1,\ldots,v_m\}$. 
Elements of $V$ are called vertices and the elements of $E$ are the  hyperedges.
Then we can associate $\Hc$ to a bipartite graph ${\rm Bip} \Hc$
 with a bipartition $V \cup E$ such that $\{v_i, e_j\}$ is an edge of ${\rm Bip} \Hc$
if $v_i \in e_j$.
Assume that ${\rm Bip} \Hc$ is connected.
A {\em hypertree} in $\Hc$ is a function ${\bf f}: E \rightarrow \ZZ_{\ge 0}$
such that there exists a spanning tree $\Gamma$ of ${\rm Bip} \Hc$ 
whose vertices have degree ${\bf f} (e) +1$ at each $e \in E$.
Then we say that $\Gamma$ induce ${\bf f}$.
Let ${\rm HT}(\Hc)$ denote the set of all hypertrees in $\Hc$.
A hyperedge $e_j \in E$ is said to be {\em internally active}
with respect to the hypertree ${\bf f}$ if it is not possible to 
decrease ${\bf f}(e_j)$ by $1$ and increase ${\bf f}(e_{j'})$ ($j' < j$) by $1$
so that another hypertree results.
We call a hyperedge {\em internally inactive} with respect to a hypertree
 if it is not internally active and denote the number of such hyperedges 
of ${\bf f}$ by $\overline{\iota} ({\bf f}) $.
Then the {\em interior polynomial} of $\Hc$
is the generating function 
$I_\Hc (x) = \sum_{{\bf f} \in {\rm HT} (\Hc)} x^{ \overline{\iota} ({\bf f}) }$.
It is known \cite[Proposition~6.1]{interior} that $\deg I_\Hc (x) \le \min\{|V|,|E|\} - 1$.
If $G = {\rm Bip} \Hc$, then we set $I_G (x) = I_\Hc (x)$.
In \cite{KalPos}  K\'{a}lm\'{a}n and Postnikov proved the following.

\begin{Proposition}[{\cite[Theorems 1.1 and 3.10]{KalPos}}]
\label{prop:interior}
Let $G$ be a connected bipartite graph. Then we have
	$$
I_G(x)  =  h^*(P_G , x).
$$
\end{Proposition}
Interior polynomials of disconnected bipartite graphs are defined by Kato \cite{KATO} and the same assertion of Proposition~\ref{prop:interior} holds for all bipartite graphs.
Note that if $G$ is a bipartite graph, then the bipartite graph $\widehat{G}$ 
arising from $G$ is connected.
Hence we can use this formula to study the equation (\ref{fundamentalF}) in Proposition~\ref{hpolyformula}.

A {\em $k$-matching} of $G$ is a set of $k$ pairwise non-adjacent edges of $G$.
Let 
$$
M(G, k)=
\left\{\{v_{i_1},\ldots , v_{i_k},  e_{j_1},\ldots,e_{j_k}\}
:
\begin{array}{c}
\mbox{there exists a } k\mbox{-matching of }
G\\
 \mbox{ whose vertex set is }
\{v_{i_1},\ldots , v_{i_k}, e_{j_1},\ldots,e_{j_k}\}
\end{array}
\right\}.
$$
For $k=0$, we set $M(G,0) = \{\emptyset\}$.
Using the theory of generalized permutohedra \cite{OhTransversal, Postnikov}, 
we have the following important fact on interior polynomials:

\begin{Proposition}
\label{kmatching_interior}
Let $G$ be a bipartite graph.
Then we have
\begin{equation}
\label{InteriorMatching}
I_{\widehat{G}} (x) = 
\sum_{k \ge 0}
| M(G, k)| \ x^k.
\end{equation}
\end{Proposition}

\begin{proof}
Let $V \cup E$ denote a bipartition of $G$,
where $V = \{v_2,\ldots, v_m\}$ and $E=\{e_2,\ldots,e_n\}$ with $d = m+n-2$.
Then $\widehat{G}$ is a connected bipartite graph
with a bipartition $V' \cup E'$ with $V' = \{v_1\} \cup V$
and $E' = \{e_1\} \cup E$.
Recall that $\{v_i, e_j\}$ is an edge of $\widehat{G}$ 
if and only if either $(i-1)(j-1) = 0$ or $\{v_i, e_j\}$ is an edge of $G$.
Let ${\rm HT}(\widehat{G})$ be the set of all hypertrees in
the hypergraph associated with $\widehat{G}$.
Given a hypertree ${\mathbf f} \in {\rm HT}(\widehat{G})$,
let $\Gamma$ be a spanning tree that induces ${\bf f}$.
From \cite[Lemma~3.3]{interior}, we may assume that 
$\{v_1, e_j\}$ is an edge of $\Gamma$ for all $1 \le j \le n$.
Note that the degree of each $v_i$ $(2 \le i \le m)$ is $1$.

By definition, $e_1$ is always internally active.
We show that, $e_j$ ($j \ge 2$) is internally active
if and only if ${\bf f} (e_j) =0$.
By definition, if ${\bf f} (e_j) =0$, then $e_j$ is internally active.
Suppose  ${\bf f} (e_j) > 0$.
Then there exists $i \ge 2$ such that 
$\{v_i, e_j\}$ is an edge of $\Gamma$.
Let ${\bf f}' \in {\rm HT}(\widehat{G})$ be a hypertree
induced by a spanning tree
obtained by replacing $\{v_i, e_j\}$ with $\{v_i, e_1\}$ in $\Gamma$.
Then we have ${\bf f}' (e_j) = {\bf f} (e_j) -1$,
${\bf f}' (e_1) = {\bf f} (e_1) +1$ and ${\bf f}' (e_k) = {\bf f} (e_k)$
for all $1 < k \neq j$. 
Hence $e_j$ is not internally active.
Thus $\overline{\iota} ({\bf f})$ is the number of
$e_j$ ($j \ge 2$) such that there exists an edge $\{v_i, e_j\}$ 
of $\Gamma$
for some $i \ge 2$.

In order to prove the equation (\ref{InteriorMatching}),
it is enough to show that,
for fixed hyperedges $e_{j_1},\ldots,e_{j_k}$ with 
$2 \le j_1 < \cdots < j_k \le n$,
the cardinality of
$$
\Sc_{j_1,\ldots,j_k}=
\{{\bf f} \in {\rm HT}(\widehat{G}) : 
 e_{j_1},\ldots,e_{j_k}
\mbox{ are not internally active and } \overline{\iota} ({\bf f}) = k
\}
$$
is equal to the cardinality of
$$
{\mathcal M}_{j_1,\ldots,j_k}=
\left\{\{v_{i_1},\ldots , v_{i_k}\}
:
\begin{array}{c}
\mbox{there exists a } k\mbox{-matching of }
G\\
 \mbox{ whose vertex set is }
\{v_{i_1},\ldots , v_{i_k}, e_{j_1},\ldots,e_{j_k}\}
\end{array}
\right\}.
$$
Let $G_{j_1,\ldots,j_k}$ be the induced subgraph of $G$ on the vertex set
$V \cup \{ e_{j_1},\ldots,e_{j_k}\}$.
If $e_{j_\ell}$ is an isolated vertex in $G_{j_1,\ldots,j_k}$, then 
both $\Sc_{j_1,\ldots,j_k}$ and ${\mathcal M}_{j_1,\ldots,j_k}$ are empty sets.
If $v_i$ is an isolated vertex in $G_{j_1,\ldots,j_k}$, then 
there is no relations between $v_i$ and two sets,
and hence we can ignore $v_i$.
Thus we may assume that $G_{j_1,\ldots,j_k}$ has no isolated vertices.

It is known that ${\mathcal M}_{j_1,\ldots,j_k}$ is the set of bases of a transversal matroid
associated with $G_{j_1,\ldots,j_k}$.
See, e.g., \cite[Section~1.6]{MT}.
For $i=2,\ldots, m$, let 
$$I_i = \{0\} \cup \{\ell : \{v_i, e_{j_\ell}\} \mbox{ is an edge of } G_{j_1,\ldots,j_k} \}
\subset \{0,1,\ldots,k\}
.$$
Oh \cite{OhTransversal} defines a lattice polytope $P_{{\mathcal M}_{j_1,\ldots,j_k}}$ to be the generalized permutohedron \cite{Postnikov}
of the induced subgraph of $\widehat{G}$ on the vertex set $V \cup \{e_1, e_{j_1},\ldots,e_{j_k}\}$, i.e., $P_{{\mathcal M}_{j_1,\ldots,j_k}}$ is the Minkowski sum 
$
\Delta_{I_2} + \cdots + \Delta_{I_m},
$
where  $\Delta_{I} = {\rm conv}(\{\eb_j : j \in I\}) \subset \RR^{k+1}$
and $\eb_0,\eb_1,\ldots,\eb_k$ are unit coordinate vectors in $\RR^{k+1}$.
By \cite[Lemma 22 and Proposition 26]{OhTransversal}, the cardinality of ${\mathcal M}_{j_1,\ldots,j_k}$ is equal to the number of the lattice point 
$(x_0,x_1,\ldots,x_k) \in P_{{\mathcal M}_{j_1,\ldots,j_k}} \cap \ZZ^{k+1}$
with $x_1, x_2 ,\ldots, x_k \ge 1$.
In addition, by \cite[Proposition~14.12]{Postnikov},
any lattice point 
$(x_0,x_1,\ldots,x_k) \in P_{{\mathcal M}_{j_1,\ldots,j_k}} \cap \ZZ^{k+1}$
is of the form
$\eb_{i_2} + \cdots + \eb_{i_m}$, where $i_\ell \in I_\ell$ for $2 \le \ell \le m$.
By a natural correspondence $(x_0,x_1,\ldots,x_k) \in P_{{\mathcal M}_{j_1,\ldots,j_k}} \cap \ZZ^{k+1}$ with $x_1, x_2 ,\ldots, x_k \ge 1$
and $({\bf  f}(e_1),{\bf  f}(e_{j_1}), \ldots, {\bf  f}(e_{j_k}))$ with 
${\bf f} \in \Sc_{j_1,\ldots,j_k}$,
it follows that  the number of the lattice point 
$(x_0,x_1,\ldots,x_k) \in P_{{\mathcal M}_{j_1,\ldots,j_k}} \cap \ZZ^{k+1}$
with $x_1, x_2 ,\ldots, x_k \ge 1$
is equal to the cardinality of $\Sc_{j_1,\ldots,j_k}$, as desired.
\end{proof}

Now, we show that the $h^*$-polynomial of $\Bc_G$ is $\gamma$-positive if $G$ is a bipartite graph.
In fact, we prove Theorem \ref{hpolymain}.
\begin{proof}[Proof of Theorem \ref{hpolymain}]
By Propositions \ref{hpolyformula} and \ref{kmatching_interior},
the $h^*$-polynomial of $\Bc_{G}$ is
\begin{eqnarray*}
h^*(\Bc_G, x)
&=&
\sum_{j=0}^d \ \ 
2^j (x-1)^{d-j}  
	\sum_{H \in S_j(G)} I_{\widehat{H}} (x)\\
&=&
\sum_{j=0}^d 
2^j (x-1)^{d-j}  
	\sum_{H \in S_j(G)} 
\sum_{k \ge 0}
|M(H, k)| \ x^k.
\end{eqnarray*}
Note that, for each 
$\{v_{i_1}, \ldots, v_{i_k},e_{j_1},\ldots,e_{j_k}\}
\in M(G, k)$ there exist $\binom{d-2k}{j -2k}$ induced subgraphs
$H \in S_j(G)$ such that $\{v_{i_1}, \ldots, v_{i_k},e_{j_1},\ldots,e_{j_k}\} \in M(H, k)$ for $j = 2k, 2k+1, \ldots, d$.
Thus we have
\begin{eqnarray*}
h^*(\Bc_G, x)
&=&
\sum_{k \ge 0}
\sum_{j=2k}^d 
2^j (x-1)^{d-j}  
| M(G, k)| 
\binom{d-2k}{j-2k}
x^k\\
&=&
\sum_{k \ge 0}
|M(G, k)| \ 
2^{2 k} x^k (2+(x-1))^{d-2 k}\\
&=&
(x+1)^d \sum_{k \ge 0}
| M(G, k) | \ 
\left(\frac{4x}{(x+1)^2}\right)^k\\
&=&  (x+1)^d I_{\widehat{G}} \left(  \frac{4x}{(x+1)^2} \right),
\end{eqnarray*}
as desired.
\end{proof}

By Proposition~\ref{kmatching_interior},
it follows that, if $G$ is a forest, then
$I_{\widehat{G}} (x) $ coincides with the {\em matching generating polynomial} of $G$.

\begin{Proposition}
\label{prop:forest}
Let $G$ be a forest.
Then we have $$I_{\widehat{G}} (x) = \sum_{k \ge 0} m_k(G) \ x^k,$$
where $m_k(G)$ is the number of $k$-matchings in $G$.
In particular, $I_{\widehat{G}} (x)$ is real-rooted.
\end{Proposition}

\begin{proof}
Let $M_1$ and $M_2$ be $k$-matchings of $G$.
Suppose that $M_1$ and $M_2$ have the same vertex set
$\{v_{i_1}, \ldots, v_{i_k},e_{j_1},\ldots,e_{j_k}\}$.
If $M=(M_1 \cup M_2) \setminus (M_1 \cap M_2)$ is not empty,
then $M$ corresponds to a subgraph of $G$ such that the degree of each vertex is $2$.
Hence $M$ has at least one cycle.
This contradicts that $G$ is a forest.
Hence we have $M_1 = M_2$.
Thus
$
m_k(G) 
$
 is the cardinality of 
$ 
 M(G, k)
$.
In general, it is known that $\sum_{k \ge 0} m_k(G) \ x^k$ is
real-rooted for any graph $G$.
See, e.g., \cite{Fact_on_matching1, Fact_on_matching2}.
\end{proof}

Next we will show that, if a bipartite graph $G$ is a ``permutation graph''
associated with a poset $P$, 
then the interior polynomial $I_{\widehat{G}} (x) $ coincides with the $P$-Eulerian polynomial $W(P)(x)$.
A  {\em permutation graph} is a graph on $[d]$ with edge set
$$\{ \{i,j\} : L_i \mbox{ and } L_j \mbox{ intersect each other} \},$$
where there are $d$ points $1,2,\ldots,d$ on two parallel lines ${\mathcal L}_1$ and ${\mathcal L}_2$ in the plane, and
the straight lines $L_i$ connect $i$ on ${\mathcal L}_1$
and $i$ on ${\mathcal L}_2$.
If $G$ is a bipartite graph with a bipartition $V_1 \cup V_2$, the following conditions are equivalent:
\begin{itemize}
\item[(i)]
$G$ is a permutation graph;
\item[(ii)]
The complement of $G$ is a comparability graph of a poset;
\item[(iii)]
There exist orderings $<_1$ on $V_1$ and $<_2$ on $V_2$ such that
$$
i, i' \in V_1, i<_1 i', \ \  j, j' \in V_2,  j <_2 j', \ \ 
 \{i,j\}, \{i', j'\} \in E(G)
\Longrightarrow
\{i,j'\}, \{i', j\} \in E(G);
$$
\item[(iv)]
For any three vertices,
there exists a pair of them such that there exists no path
containing the two vertices that avoids the neighborhood of
the remaining vertex.
\end{itemize}
See \cite[p.~93]{graphclasses} for details.
On the other hand,
let $P$ be a naturally labeled poset $P$ on $[d]$.
Then the {\em order polynomial} $\Omega(P,m)$ of $P$
is defined for $0 < m \in \ZZ$ to be the number of order-preserving maps
$\sigma: P \rightarrow [m]$.
It is known that 
$$
\sum_{m \ge 0} \Omega(P,m+1) x^m
=
\frac{ \sum_{\pi \in {\mathcal L} (P)} x^{d(\pi)}}{ (1-x)^{d+1} },
$$
where ${\mathcal L} (P)$ is the set of linear extensions of $P$
and $d(\pi)$ is the number of descent of $\pi$.
The {\em $P$-Eulerian polynomial} $W(P)(x)$ is defined by
$$
W(P)(x)
=
 \sum_{\pi \in {\mathcal L} (P)} x^{d(\pi)}
.$$
See, e.g., \cite{Ebook} for details.
We now see a relation between the interior polynomial of a bipartite permutation graph and the $P$-Eulerian polynomial of a finite poset.

\begin{Proposition}
	\label{prop:permutation}
	Let $G$ be a bipartite permutation graph and let $P$
be a poset whose comparability graph is the complement of 
$G$.
Then we have $$I_{\widehat{G}} (x) = W(P)(x).$$
\end{Proposition}

\begin{proof}
In this case, $\Bc_G \cap {\mathcal O}_{(1,\ldots,1)}$ is
the chain polytope ${\mathcal C}_{P}$ of $P$.
It is known that the $h^*$-polynomial of ${\mathcal C}_{P}$
is the $P$-Eulerian polynomial $W(P)(x)$.
See \cite{twoposetpolytopes, Ebook} for details.
Thus we have $I_{\widehat{G}} (x) = h^*(P_{\widehat{G}}, x)= W(P)(x),$
as desired.
\end{proof}

It was conjectured by
Neggers--Stanley that $W(P)(x)$ is real-rooted.
However this is false in general.
The first counterexample was constructed in \cite{Bra} (not naturally labeled posets).
Counterexamples of naturally labeled posets were given in \cite{Stembridge}.
Counterexamples in these two papers are {\em narrow} posets,
i.e., elements of posets are partitioned into two chains.
It is easy to see that $P$ is narrow poset if and only if
the comparability graph of $P$ is the complement of a bipartite graph.
Since Stembridge found many counterexamples which are naturally labeled narrow posets,
there are many bipartite permutation graphs $G$ such that $h^*(\Bc_G ,x)$ are not real-rooted.
We give one of them as follows.

\begin{Example}
	\label{ex:nonreal}
Let $P$ be a naturally labeled poset in Figure~\ref{Sposet} given in
\cite{Stembridge}.
\begin{figure}[h]
\begin{center}
\includegraphics[height=6cm]{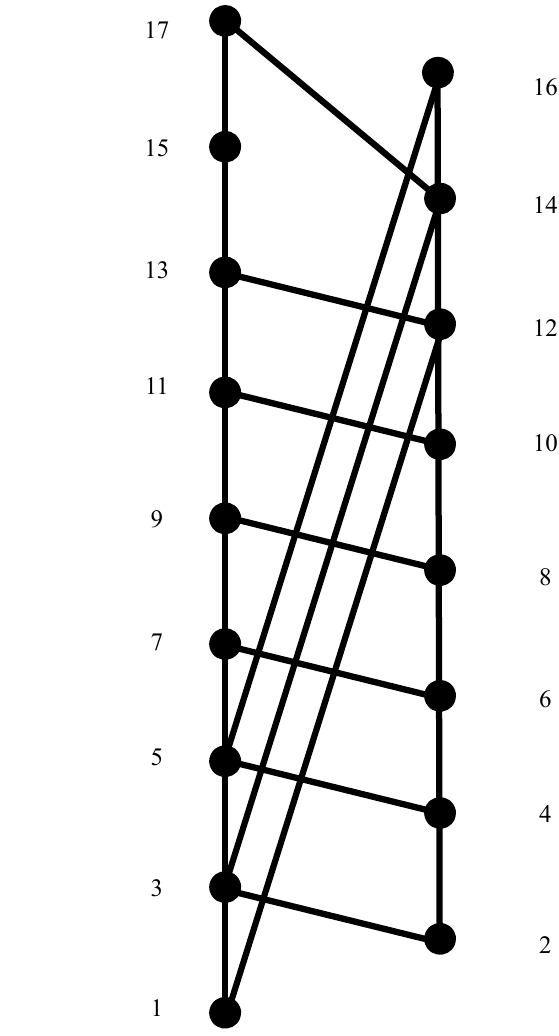} 
\caption{A counterexample to Neggers--Stanley conjecture \cite{Stembridge}}
\label{Sposet}
\end{center}
\end{figure}
Then $$W(P)(x)=
3 x^8+86 x^7+658 x^6+1946 x^5+2534 x^4+1420 x^3+336 x^2+32 x+1$$
has a conjugate pair of zeros near $-1.85884 \pm 0.149768 i$ as explained in \cite{Stembridge}.
Let $G$ be the complement of the comparability graph of $P$.
Then $G$ is a bipartite graph with $17$ vertices and $32$ edges.
The $h^*$-polynomial of $\Bc_G$ is 
\begin{eqnarray*}
& & h^*(\Bc_G ,x)= (x+1)^{17} W(P)\left(\frac{4x}{(x+1)^2}\right)\\
&=&
x^{17}+145 x^{16}+7432 x^{15}+174888 x^{14}+2128332 x^{13}+14547884 x^{12}+59233240 x^{11}\\
& & +148792184 x^{10}+234916470 x^9+234916470 x^8+148792184 x^7+59233240 x^6\\
& & +14547884 x^5+2128332 x^4+174888 x^3+7432 x^2+145 x+1
\end{eqnarray*}
and 
has conjugate pairs of zeros near
$-3.88091 \pm 0.18448 i$ and $-0.257091\pm0.0122209 i$.
(We used {\tt Mathematica} to compute approximate values.)
On the other hand, $h^*(\Bc_G ,x)$ is log-concave.
\end{Example}

By the following proposition, it turns out that
this example is a counterexample to ``Real Root Conjecture'' that has been already disproved by Gal \cite{Gal}.

\begin{Proposition}
\label{sphere}
	Let $G$ be a bipartite permutation graph.
	Then $h^*(\Bc_G, x)$ coincides with the $h$-polynomial of  a flag complex
	that is a triangulation of a sphere.
\end{Proposition}

\begin{proof}
It is known \cite[p.94]{graphclasses} that any bipartite permutation graph satisfies the condition (iv) in 
Theorem~\ref{kb}.
Hence there exists a squarefree quadratic initial ideal with respect to 
a reverse lexicographic order such that the smallest variable
corresponds to the origin.
This means that there exists a flag regular unimodular triangulation $\Delta$
such that the origin is a vertex of any maximal simplex in $\Delta$.
Then $h^*(\Bc_G, x)$ coincides with the $h$-polynomial of a 
flag triangulation of the boundary of 
a convex polytope $\Bc_G$ arising from $\Delta$.
\end{proof}

In \cite{HJMsymmetric}, for a $(p,q)$-complete bipartite graph $K_{p,q}$, a simple description for the $h^*$-polynomial of $\Ac_{K_{p,q}}$ was given. 
In fact, one has
\begin{equation}
\label{formula:A}
h^*(\Ac_{K_{p+1,q+1}}, x)
=
\sum_{i = 0}^{\min\{p,q\}}
\binom{2i}{i}  \binom{p}{i} \binom{q}{i} x^i (x+1)^{p+q+1 - 2 i}.
\end{equation}
Moreover, it was shown that $h^*(\Ac_{K_{p+1,q+1}}, x)$ is  $\gamma$-positive and real-rooted.
Similarly, we can obtain a simple description for the $h^*$-polynomial of $\Bc_{K_{p,q}}$  and show that  $h^*(\Bc_{K_{p,q}}, x)$ is $\gamma$-positive and real-rooted.
\begin{Example}
	\label{ex:complete}
	Let $K_{p,q}$ be a $(p,q)$-complete bipartite graph.
	Then the comparability graph of a poset $P$
	consisting of two disjoint chains
	$1 < 2 < \cdots < p$ and $p+1 < p+2 < \cdots < p+q$
	is the complement of $K_{p,q}$.
	It is easy to see that
	$ $
	$$
	W(P)(x) = \sum_{i = 0}^{\min\{p,q\}}
	\binom{p}{i} \binom{q}{i} x^i.
	$$
	Hence we have
	\begin{equation}
	\label{formula:B}
	h^*(\Bc_{K_{p,q}}, x)
	=
	(x+1)^{p+q} W(P) \left(\frac{4x}{(x+1)^2} \right)
	=
	\sum_{i = 0}^{\min\{p,q\}}
	4^i  \binom{p}{i} \binom{q}{i} x^i (x+1)^{p+q - 2 i}.	
	\end{equation}
	Simion \cite{Simion} proved that $W(P)(x)$ is real-rooted if 
	$P$ is a naturally labeled and disjoint union of chains.
	Thus the $h^*$-polynomial $h^*(\Bc_{K_{p,q}}, x)$ is real-rooted.	
\end{Example}

\begin{Remark}
In a very recent work \cite{locallyanti}, it is shown that the $h^*$-polynomials of 
symmetric edge polytopes of certain graphs can be computed by that of 
$\Bc_G$.
For example, the equation (\ref{formula:A}) can be obtained from the equation (\ref{formula:B})
(\cite[Example. 4.10]{locallyanti}).
\end{Remark}

In \cite{HJMsymmetric}, for the proof of the real-rootedness of $h^*(\Ac_{K_{p,q}},x)$, interlacing polynomials techniques were used.
Let $f$ and $g$ be real-rooted polynomials with roots $a_1\geq a_2 \geq \cdots$, respectively, $b_1 \geq b_2 \geq \cdots$.
Then $g$ is said to {\em interlace} $f$ if
\[
a_1 \geq b_1 \geq a_2 \geq b_2 \geq \cdots.
\] 
In this case, we write $f \preceq g$.
In \cite{HJMsymmetric}, it is shown that 
\[h^*(\Ac_{K_{p,q}},x) \preceq h^*(\Ac_{K_{p,q+1}},x).\]

By a similar way of \cite{HJMsymmetric}, we can prove the following.
\begin{Proposition}
	For all $p, q \geq 1$, one has \[h^*(\Bc_{K_{p,q}},x) \preceq h^*(\Bc_{K_{p,q+1}},x).\]
\end{Proposition}
\begin{proof}
	Set $\gamma(\Bc_{K_{p,q}},x)=\sum_{i \geq 0}
	4^i  \binom{p}{i} \binom{q}{i} x^i$.
	Since $\{\binom{p}{i}\}_{i \geq 0}$ is a multiplier sequence (see \cite{HJMsymmetric}) and since $(4x+1)^{q} \preceq (4x+1)^{q+1}$, one has 
	$\gamma(\Bc_{K_{p,q}},x) \preceq \gamma(\Bc_{K_{p,q+1}},x).$
	By \cite[Lemma 4.10]{HJMsymmetric}, we obtain
	$h^*(\Bc_{K_{p,q}},x) \preceq h^*(\Bc_{K_{p,q+1}},x).$
\end{proof}

\end{document}